\documentclass[11pt]{amsart}
\usepackage{latexsym}
\usepackage{amsfonts}

\usepackage{amsmath,amsthm,amssymb}

\newtheorem{theor}{Theorem}

\newtheorem{thm}{Theorem}[section]
\newtheorem{lem}[thm]{Lemma}
\newtheorem{cor}[thm]{Corollary}

\newtheorem{prop}[thm]{Proposition}
\theoremstyle{definition}
\newtheorem{defn}[thm]{Definition}
\newtheorem{notation}[thm]{Notation}
\newtheorem{conv}[thm]{Convention}
\newtheorem{rem}[thm]{Remark}

\newtheorem{propdfn}[thm]{Proposition-Definition}

\def\strutdepth{\dp\strutbox}
\def \ss{\strut\vadjust{\kern-\strutdepth \sss}}
\def \sss{\vtop to \strutdepth{
\baselineskip\strutdepth\vss\llap{$\diamondsuit\;\;$}\null}}

\def\strutdepth{\dp\strutbox}
\def \sst{\strut\vadjust{\kern-\strutdepth \ssss}}
\def \ssss{\vtop to \strutdepth{
\baselineskip\strutdepth\vss\llap{$\spadesuit\;\;$}\null}}

\def\strutdepth{\dp\strutbox}
\def \ssh{\strut\vadjust{\kern-\strutdepth \sssh}}
\def \sssh{\vtop to \strutdepth{
\baselineskip\strutdepth\vss\llap{$\heartsuit\;\;$}\null}}

\newcommand{\R}{\mathbb R}
\newcommand{\Z}{\mathbb Z}

\begin{document}

\title[Geometric Intersection Number]{Geometric Intersection Number and analogues of the Curve Complex for free
groups}

\author[I.~Kapovich]{Ilya Kapovich}

\address{\tt Department of Mathematics, University of Illinois at
Urbana-Champaign, 1409 West Green Street, Urbana, IL 61801, USA
\newline http://www.math.uiuc.edu/\~{}kapovich/} \email{\tt
kapovich@math.uiuc.edu}

\author[M.~Lustig]{Martin Lustig}\address{\tt Math\'ematiques (LATP),
Universit\'e Paul C\'ezanne - Aix Marseille III\\ ave.  Escadrille
Normandie-Ni\'emen, 13397 Marseille 20, France} \email{\tt
Martin.Lustig@univ-cezanne.fr}

\begin{abstract}

For the free group $F_{N}$ of finite rank $N \geq 2$ we construct a canonical Bonahon-type,
continuous and $Out(F_N)$-invariant \emph{geometric intersection form}
$$\langle \ ,\ \rangle:
\overline{cv}(F_N)\times Curr(F_N)\to \mathbb R_{\ge 0}.$$

Here $\overline{cv}(F_N)$ is the closure of
unprojectivized Culler-Vogtmann's Outer space $cv(F_N)$ in the equivariant
Gromov-Hausdorff convergence topology (or, equivalently, in the length
function topology). It is known that $\overline{cv}(F_N)$ consists of
all \emph{very small} minimal isometric actions of $F_N$ on $\mathbb
R$-trees.  The projectivization of $\overline{cv}(F_N)$ provides a
free group analogue of Thurston's compactification of the Teichm\"uller
space.

As an application, using the \emph{intersection graph} determined by
the intersection form, we show that several natural analogues of the
curve complex in the free group context have infinite diameter.
\end{abstract}

\thanks{The first author was supported by the NSF grants DMS-0404991
and DMS-0603921.  Both authors acknowledge the support of MSRI
Berkeley a part of the special semester in Geometric Group Theory}


\maketitle

\section{Introduction}

The notion of an intersection number plays a crucial role in the study
of Teichm\"uller space, mapping class groups, and their
applications to 3-manifold topology.  Thurston~\cite{Th} extended the
notion of a geometric intersection number between two free homotopy
classes of closed curves on a surface to the notion of a
{\em geometric intersection number} between two measured geodesic laminations.  Indeed, this
intersection number is a central concept in
the study of Thurston's compactification of the Teichm\"uller space, as
well as in the study of the dynamics and geometry of surface homeomorphisms.
Bonahon~\cite{Bo86} extended this notion of geometric intersection number to
the case of two geodesic currents on the surface.  Bonahon also
constructed~\cite{Bo88} a mapping class group equivariant embedding of
Thurston's compactification of the Teichm\"uller space into the space
of projectivized geodesic currents.

Culler and Vogtmann introduced in \cite{CV} a free-group analogue of
Teichm\"uller space, which has been termed {\em Outer space} by Shalen and is denoted here by
$CV(F_N)$ (where $F_N$ is a free group of finite rank $N\ge 2$).
Whereas points in Teichm\"uller space can be thought of as free
and discrete isometric actions of the surface group on $\mathbb H^2$,
points in $CV(F_N)$ are represented by minimal
free and discrete isometric actions of $F_N$ on $\mathbb R$-trees
with normalized co-volume (that is, where the quotient graph has volume $1$). One also often works
with the \emph{unprojectivized Outer space} $cv(F_N)$,
which contains a copy of $CV(F_{N})$, and consists of all actions of the above type with arbitrary co-volume.
More details are given in Section~\ref{sect:cv} below.

\smallskip

Let $\overline{cv}(F_N)$ be the
closure of $cv(F_N)$ in the equivariant Gromov-Hausdorff topology. It
is known~\cite{CL,BF93} that $\overline{cv}(F_N)$ consists precisely
of all the minimal \emph{very small} nontrivial isometric actions on $F_N$ on
$\mathbb R$-trees (see Section~\ref{sect:cv} for definitions). Projectivizing $\overline{cv}(F_N)$ gives a Thurston-type compactification
$\overline{CV}(F_{N})=CV(F_N)\cup \partial CV(F_N)$ of Outer Space, where $\partial CV(F_N)$ is the \emph{Thurston boundary} of $CV(F_N)$.
The outer automorphism group $Out(F_{N})$ of $F_{N}$ acts on ${CV}(F_{N})$ and
$\overline{CV}(F_{N})$ in very close analogy
to the action of the mapping class group on Teichm\"uller
space and its Thurston compactification. One can regard $\partial cv(F_N)=\overline{cv}(F_N)-cv(F_N)$ as the \emph{boundary} of $cv(F_N)$.
The Thurston boundary $\partial CV(F_N)$ is obtained by projectivizing $\partial cv(F_N)$.

\smallskip

The structure of Outer space and of $Out(F_N)$ is more complicated
than that of the Teichm\"uller space and the mapping class
group. In large part this is due to the fact
most free group automorphisms are not ``geometric'', in the
sense that they are not induced by surface homeomorphisms.
Although finite dimensional, $CV(F_{N})$ is not a manifold, and
hence none of the tools from complex analysis which are so useful for
surfaces can be directly carried over into the free group world.
Moreover, while the topological and homotopy properties of  Outer space are
fairly well understood, very little is known about the geometry of
$CV(F_N)$.  One of the reasons for this has been the lack, until
recently, of a good geometric intersection theory in the Outer space context.

\smallskip

A \emph{geodesic current} is a measure-theoretic generalization of the
notion of a conjugacy class of a group element or of a free homotopy
class of a closed curve on a surface (see Definition~\ref{defn:current} below).  Much of the motivation for
studying currents comes from the work of Bonahon about geodesic
currents on hyperbolic surfaces~\cite{Bo86,Bo88}.  The space
$Curr(F_N)$ of all geodesic currents has a useful linear structure and
admits a canonical $Out(F_N)$-action.  The space $Curr(F_N)$ turns out
to be a natural companion of the Outer space and contains additional
valuable information about the geometry and dynamics of free group
automorphisms.  Examples of such applications can be found in
\cite{Bo91,CHL3,Fra,Ka1,Ka2,Ka3,KL1,KN,KKS,Ma} and other sources.

\smallskip

In \cite{Ka2,L} we introduced a Bonahon-type, continuous, and
$Out(F_N)$-invariant \emph{geometric intersection form}
$$\langle \ ,\ \rangle:
cv(F_N)\times Curr(F_N)\to \mathbb R_{\ge 0}.$$
The geometric intersection form $\langle \ , \ \rangle$
is $\mathbb R_{>0}$-homogeneous with respect to the first argument,
$\mathbb R_{\ge 0}$-linear with respect to the second argument and is $Out(F_N)$-equivariant. This
intersection form has the following crucial property in common with
Bonahon's notion of an intersection number between two geodesic
currents on a surface:

For any $\mathbb R$-tree $T\in cv(F_N)$ and for any
$g\in F_N \smallsetminus \{1\}$ we have $\langle T, \eta_g\rangle =||g||_T$.  Here
$\eta_g$ is the \emph{counting current} of $g$ (see
Definition~\ref{defn:rational}) and $||g||_T$ is the translation
length of $g$ on the $\mathbb R$-tree $T$.  Since the scalar multiples
of all counting currents form a dense set in $Curr(F_N)$, there is at
most one continuous intersection form with the above properties, so
that $\langle \ , \ \rangle$ is in fact canonical.
Kapovich proved~\cite{Ka2} that the intersection form $\langle \ , \ \rangle$
does not admit a ``reasonable" continuous $Out(F_N)$-invariant
symmetric extension to a map $Curr(F_N)\times Curr(F_N)\to\mathbb R$.

The main result of this paper is that the geometric intersection form $\langle \ , \ \rangle$ admits a continuous extension to the boundary of $cv(F)$. We present a simplified form of this result here and refer to Theorem~\ref{thm:int-form} below
for a more detailed statement.

\begin{theor}\label{thm:A}
Let $N\ge 2$.  There exists a unique continuous map \[\langle \ , \
\rangle: \overline{cv}(F_N)\times Curr(F_N)\to \mathbb R_{\ge 0}\]
which is $R_{\ge 0}$-homogeneous in the first argument, $R_{\ge 0}$-linear in
 the second argument, $Out(F_N)$-invariant, and such that for every $T\in \overline{cv}(F_N)$ and every $g\in F_N\smallsetminus \{1\}$ we have
 \[
\langle T, \eta_g\rangle=||g||_T.
 \]
\end{theor}

It is easy to see that the map $\langle \
, \ \rangle$ in Theorem~\ref{thm:A} coincides with the intersection
form from~\cite{Ka2}, when restricted to ${cv}(F_N)\times Curr(F_N)$.

\smallskip

Note that a very different and symmetric notion of an intersection
number between two elements of $\overline{cv}(F_N)$ was introduced and
studied by Guirardel~\cite{Gui2}.  However, Guiradel's intersection
number often takes on the value $\infty$, and it is fairly difficult to use.

\smallskip

A key ingredient in the proof of Theorem~\ref{thm:A} is
Proposition~\ref{prop:lgl} below, which establishes a ``Uniform Scaling
Approximation Property" for points in $\overline{cv}(F_N)$.  It is
clear that Proposition~\ref{prop:lgl} should have further useful
applications in the study of the boundary of the Outer space.  The
proof of Proposition~\ref{prop:lgl} in turn relies on the
\emph{Bounded Back-Tracking Property} for very small actions of $F_N$
on $\mathbb R$-trees, established by Gaboriau, Jaeger, Levitt, and
Lustig in \cite{GJLL}.

\medskip

Recall that for a closed hyperbolic
surface $S$ the \emph{curve graph} $\mathcal C(S)$ is defined as
follows.  The vertices of $\mathcal C(S)$ are free homotopy classes of
essential simple closed curves on $S$.  Two distinct vertices of
$\mathcal C(S)$ are adjacent in $\mathcal C(S)$ if they can be
realized by disjoint simple closed curves.  The curve graph turned out
to be a valuable tool in the study of the mapping class groups, of
Kleinian groups, and in various applications to 3-manifolds.
Masur-Minsky~\cite{MM} and Hempel~\cite{He} established that the curve
graph has infinite diameter.  Moreover, Masur-Minsky~\cite{MM} proved
that the curve graph is Gromov hyperbolic.

\smallskip

 Algebraically, an essential simple closed curve $\alpha$ on $S$
 determines a splitting of $G=\pi_1(S)$ as an amalgamated free product
 or an HNN-extension over the cyclic subgroup generated by $\alpha$
 (the amalgamated free product case occurs if $\alpha$ is separating
 and the HNN-extension case occurs if $\alpha$ is non-separating).
 Moreover, it is known~\cite{ZVC} that all splittings of $G$ over $\mathbb Z$ arise in this fashion.

 \smallskip

 In the free group context, it is often more natural to consider
 splittings over the trivial group rather than over $\mathbb Z$.  Thus
 we define the \emph{free splitting graph} $\mathcal
 F=\mathcal{F}(F_N)$ as follows.  The vertices of $\mathcal F$
 correspond to proper free product decompositions $F_N=A\ast B$, where
 $A\ne \{1\}, B\ne \{1\}$, where two such splittings are considered to
 be equal if their Bass-Serre trees are $F_N$-equivariantly isometric.
 Adjcency in $\mathcal{F}(F_N)$ corresponds to admitting a splitting
 of $F_N$ that is a common refinement of the two splittings in
 question. Informally, two distinct splittings $F_N=A\ast B$ and
 $F_N=A'\ast B'$ are adjacent in $\mathcal F(F_N)$ if there exists a
 free product decomposition $F_N=C_1\ast C_2\ast C_3$ such that either
 $A=C_1\ast C_2$, $B=C_3$ and $A'=C_1$, $B'=C_2\ast C_3$, or else
 $A'=C_1\ast C_2$, $B'=C_3$ and $A=C_1$, $B=C_2\ast C_3$. It is not
 hard to see that $\mathcal F_N$ is connected for $N\ge 3$. We also
 define the \emph{dual free splitting graph} $\mathcal
 F^\ast=\mathcal{F}^\ast(F_N)$ as follows. The vertex set of
 $\mathcal{F}^\ast(F_N)$ is the same as the vertex set of
 $\mathcal{F}(F_N)$. Two vertices $T_1$ and $T_2$ of
 $\mathcal{F}^\ast(F_N)$ are adjacent in $\mathcal{F}^\ast(F_N)$ if
 there exists a nontrivial element $a\in F_N$ which is elliptic with
 respect to both $T$ and $T_1$, that is $||a||_{T_1}=||a||_{T_2}$
 (this adjacency condition turns out be be equivalent to requiring
 that there exist a primitive, i.e. a member of a free basis, element
 of $F_N$ that is elliptic for both $T_1$ and $T_2$). In the standard
 curve complex context, analogues of definitions of adjacency in
 $\mathcal F$ and $\mathcal F^\ast$ are essentially equivalent to the
 standard definition of adjacency in $\mathcal C(S)$. Namely, two
 non-isotopic simple closed curves define adjacent vertices of
 $\mathcal C(S)$ if and only if the corresponding cyclic splittings of
 $\pi_1(S)$ have a common refinement. Also, two such curves define
 vertices at distance $\le 2$ in $\mathcal C(S)$ if and only if the
 corresponding splittings of $\pi_1(S)$ have a common nontrivial
 elliptic element. However, in the context of free groups $\mathcal F$
 and $\mathcal F^\ast$ appear to be rather different objects, with
 distances in $\mathcal F^\ast$ often being much smaller than in
 $\mathcal F$.

 We also introduce a key new object  $\mathcal I(F_N)$ called the
 \emph{intersection graph} of $F_N$.  The graph $\mathcal I(F_N)$ is a
 bipartite graph with vertices of two kinds: projective classes $[T]$
 of very small $F_N$-trees $T\in \overline{cv}(F_N)$ and projective
 classes $[\mu]$ of nonzero currents $\mu\in Curr(F_N)$. Two vertices
 $[T]$ and $[\mu]$ are adjacent in $\mathcal I(F_N)$ whenever $\langle
 T,\mu\rangle=0$. For $N\ge 3$ the graph $\mathcal I(F_N)$ has a large
 $Out(F_N)$-invariant connected component $\mathcal I_0(F_N)$ that contains all projective
 classes of Bass-Serre trees $T$ corresponding to nontrivial free
 product decompositions of $F_N$ (this component also contains all the
 projective currents $[\eta_a]$ corresponding to primitive elements
 $a$ of $F_N$). Both $\mathcal F(F_N)$ and $\mathcal F^\ast(F_N)$ admit $Out(F_N)$-equivariant
 Lipshitz maps into $\mathcal I_0(F_N)$ as do essentially all other reasonable
 analogues of the notion of a curve complex for free groups.

\smallskip

The results proved in sections 7 and 8
of this paper can be summarized as follows:

\begin{theor}\label{thm:B}
Let $N\ge 3$. Then the graphs $\mathcal I_0(F_N)$,
$\mathcal F(F_N)$ and $\mathcal F^*(F_N)$
have infinite diameter.

Moreover, if $Y_N$ is one of the above graphs and $\phi\in Out(F_N)$ is 
an atoroidal iwip, i.e.  $\phi$ is irreducible with irreducible powers (see Definition~\ref{defn:iwip})
and has no periodic conjugacy classes in $F_N$, then for any vertices $x,y$ of $Y_N$ we have
\[
\lim_{n\to\infty} d_{Y_N}(x,\phi^n y)=\infty.
\]
\end{theor}

Recently Behrstock, Bestvina and Clay \cite{BBC} obtained
by different arguments an independent proof the conclusion of Theorem~\ref{thm:B} for the complex  $\mathcal S(F_N)$ which is quasi-isometric to $\mathcal F(F_N)$ (see Definition \ref{othercomplexes} and the subsequent discussion). 

Our proof of Theorem~\ref{thm:B} also shows that the ``directions to infinity", given by Theorem~\ref{thm:B}, corresponding to substantially different $\phi$, are distinct. Thus one can also show that if $\psi$ and $\phi$ are two elements of $Out(F_N)$ such that they are irreducible with irreducible powers and without periodic conjugacy classes, and such that the subgroup $\langle \phi,\psi\rangle$ is not virtually cyclic then for any $x\in VY_N$ and any sequences $n_i\to\infty$, $m_i\to\infty$ we have
\[
\lim_{i\to\infty}d_{Y_N}(\phi^{n_i}x,\psi^{m_i}x)=\infty.
\]

We also consider several natural variations of $\mathcal F(F_N)$ and $\mathcal
F^*(F_N)$ and note that conclusions of Theorem~\ref{thm:B} apply to
them as well. Note that the intersection graph $\mathcal I(F_N)$ is
not connected and it has other interesting connected components apart
from $I_0(F_N)$. For example, for any $T\in cv(F_N)$ and for any
current $\mu$ with full support, $[T]$ and $[\mu]$ are
isolated vertices of $\mathcal I(F_N)$. Similarly, if $\phi\in Out(F_N)$ is an
atoroidal iwip then for the ``stable eigentree'' $T(\phi)$ and
``stable eigencurrent'' $\mu(\phi)$ the pair $[T(\phi)], [\mu(\phi)]$
forms an isolated edge in $\mathcal I(F_N)$ (see \cite{KL2} for details). 

\smallskip

It seems plausible that the graphs $\mathcal F(F_N)$ and $\mathcal F^*(F_N)$ are not quasi-isometric. Investigating hyperbolicity properties of these graphs remains an interesting open problem for future study. It appears that $\mathcal F(F_N)$ is unlikely to be Gromov-hyperbolic  while $\mathcal F^*(F_N)$ does seem to have a chance for hyperbolicity. In particular, a reducible element of $Out(F_N)$ always acts on $\mathcal F^*(F_N)$ with a bounded orbit while it seems possible for a reducible automorphism to act on $\mathcal F(F_N)$ with an unbounded orbit. Also, it seems that the orbits in $\mathcal F(F_N)$ of free abelian subgroups of $Out(F_N)$  may produce quasi-flats, while similar orbits in $\mathcal F^*(F_N)$ have finite diameter.

\medskip

The main result of this paper, Theorem~\ref{thm:A}, seems to have
the potential to become a important tool in the study of Outer space
and $Out(F_{N})$:

\begin{itemize}
\item
We prove in sections 7 and 8 of this paper
that several free group analogues of the curve
complex have infinite diameter.
\item
Ursula Hamenst\"adt had recently used~\cite{Ha} Theorem~\ref{thm:A}
as a crucial ingredient to prove that any non-elementary subgroup of
$Out(F_N)$, where $N\ge 3$, has infinite dimensional second bounded
cohomology group (infinite dimensional space of quasi-morphisms). This
in turn has an application to proving that any homomorphism from any
lattice in a higher-rank semi-simple Lie group to $Out(F_N)$ has finite image.
\smallskip
\item
Very recently Bestvina and Feighn~\cite{BF08} used Theorem~\ref{thm:A} as a key tool in the proof that for any finite collection $\phi_1,\dots,\phi_m\in Out(F_N)$ of iwip (irreducible with irreducible powers, see Definition~\ref{defn:iwip}) outer automorphisms of $F_N$ there exists a $\delta$-hyperbolic complex $X$ with an isometric $Out(F_N)$-action where each $\phi_i$ acts with a positive translation length. 
Unlike the curve complex analogues discussed here in section 7 and 8,
the Bestvina-Feighn construction is not intrinsically defined, but
their result gives substantial hope and significant indication that
some of these other more functorial and intrinsic analogues of the
curve complex for free groups may be Gromov-hyperbolic as well.

\item The results of the new paper~\cite{BF08} of Bestvina and Feighn also imply that if $\phi\in Out(F_N)$ is an iwip, then $\phi$ acts
with positive asymptotic translation length on $\mathcal F(F_N)$ and
on $\mathcal F^*(F_N)$. This means that 
when $Y_N$ is one of these two graphs, and $T\in
VY_N$ is an arbitrary vertex, then the orbit map $\mathbb Z\to Y_N,
n\mapsto \phi^n T$, is a quasi-isometric embedding.

\item
In a new preprint \cite{KL3}, we use Theorem~\ref{thm:A} to construct
domains of discontinuity for the action of subgroups of $Out(F_N)$ on
$\overline{CV}(F_N)$ and on $\mathbb P Curr(F_N)$.
\item
In another new preprint \cite{KL4} we show that every subgroup of
$Out(F_N)$ which contains an atoroidal iwip and is not virtually
cyclic, also contains a non-abelian free subgroup where every non-trivial element is an atoroidal iwip. 
\smallskip
\item
Finally, in \cite{KL2} we use Theorem~\ref{thm:A} to
characterize the situation where $\langle T,\mu\rangle=0$.
Specifically, we prove in \cite{KL2} that for $T\in
\overline{cv}(F_N)$ and $\mu\in Curr(F_N)$ we have $\langle
T,\mu\rangle=0$ if and only if $supp(\mu)\subseteq L^2(T)$.  Here
$supp(\mu)$ is the support of $\mu$ and $L^2(T)$ is the \emph{dual
algebraic lamination} of $T$ (see \cite{CHL1}).  That result in turn
is applied in \cite{KL2} to the notions of a \emph{filling conjugacy
class} and a \emph{filling current} as well as to obtain results about
\emph{bounded translation equivalence} in $F_N$.  In \cite{KL2} we
also obtain a generalization of the length compactness result of
Francaviglia~\cite{Fra}: we show that if $T\in cv(F_N)$ and
$\mu\in Curr(F_N)$ is a current with full support (e.g. the
Patterson-Sullivan current~\cite{KN}) then for every $C>0$ the set
$\{\phi\in Out(F_N): \langle T, \phi\mu\rangle \le C\}$
is finite and hence the set $\{ \langle T, \phi\mu\rangle : \phi\in
Out(F_N)\}\subseteq \mathbb R$ is discrete.
\end{itemize}

\medskip

\noindent
{\em Acknowledgments:}
The authors are grateful to Gilbert Levitt, Thierry Coulbois, Mark
Sapir and Saul Schleimer for useful and stimulating conversations that
were very helpful in writing this paper. Our special thanks go to Saul
Schleimer whose persistent questions inspired sections 7 and 8
of this paper.
The backwards iteration idea for the proof of
Proposition~\ref{prop:key} was suggested to us by Gilbert Levitt.

This paper grew out of a series of discussions
that the authors had in October-November 2007 at MSRI Berkeley.
We would like to thank MSRI and the organizers of the MSRI special semester in
Geometric Group Theory for the financial support that made these
discussions possible.

\section{Outer space and its closure}\label{sect:cv}

We will only briefly recall the main definitions related to
Outer space here.  For a more detailed background information we refer the
reader to~\cite{BF00,CL,CV,Gui1, Lu2} and other sources.

\smallskip

Let $F_N$ be a free group of finite rank $N\ge 2$.
Let $T$ be an $\mathbb R$-tree with an isometric action of
$F_{N}$. For any $g\in F_{N}$, denote
\[
||g||_T=\inf_{x\in T} d_T(x,gx)=\min_{x\in T} d_T(x,gx).
\]
The number $||g||_T$ is called the \emph{translation length} of $g$.

\begin{rem}\label{powers}
Note that for all $m \in \Z$ we have:
$$||g^m||_{T} = |m| \cdot ||g||_{T}.$$
\end{rem}

An isometric action of $F_N$ on an $\mathbb R$-tree $T$ is called
{\em minimal} if $T$ has no proper $F_{N}$-invariant subtrees.

\begin{defn}
An isometric action of $F_N$ on an $\mathbb R$-tree $T$ action is called
\emph{very small} if:
\begin{enumerate}
\item The stabilizer in $F_N$ of every non-degenerate arc in $T$ is cyclic
(either trivial or infinite cyclic).
\item The stabilizer in $F_N$ of
every non-degenerate tripod is trivial.
\item For every $g\in F_N,g\ne 1$ and every integer $n\ne 0$ if $g^n$ fixes some non-degenerate arc, then $g$ fixes that arc.
\end{enumerate}
\end{defn}

Thus free isometric actions of $F_N$ on $\mathbb R$-trees, and, more
generally, actions with trivial arc stabilizers, are very small.

\begin{defn}[Outer space and its closure]
Let $N\ge 2$ be an integer.
\smallskip
\noindent

\begin{enumerate}
\item We denote by $cv(F_N)$ the space of all minimal free and discrete
isometric actions of $F_N$ on $\mathbb R$-trees.  Two such actions of
$F_N$ on trees $T$ and $T'$ are identified in $cv(F_N)$ if there
exists an $F_N$-equivariant isometry between $T$ and $T'$.  The space
$cv(F_N)$ is called \emph{unprojectivized Outer space}
for $F_N$.

\item  Denote by $\overline{cv}(F_N)$ the space of all minimal nontrivial
very small isometric actions of $F_N$ on $\mathbb R$-trees.  Again, two such
actions are considered equal in $\overline{cv}(F_N)$ if there exists
an $F_N$-equivariant isometry between the two trees in question.
\end{enumerate}

\end{defn}

Note that if $T\in cv(F_N)$ then the quotient $T/F_N$ is compact.
It is known that every element $T\in \overline{cv}(F_N)$
is uniquely identified by its
\emph{translation length function}
$F_N\to \mathbb R, \,\,
g \mapsto ||g||_T$.
That is,
for $T,T'\in \overline{cv}(F_N)$ we have $T=T'$ if and only if
$||g||_T = ||g||_{T'}$ for all $g \in F_{n}$.

\smallskip

The spaces $cv(F_N)$ and $\overline{cv}(F_N)$ have several natural
topologies that are all known to coincide~\cite{Pau}: the pointwise
translation length function convergence topology, the equivariant
Gromov-Hausdorff-Paulin
convergence topology and the weak $CW$-topology (for the case of
$cv(F_N)$).  In particular if $T_n,T\in \overline{cv}(F_N)$ then
$\underset{n \to \infty}{\lim} T_n=T$ if and only if for every $g\in F_N$ we have
$\underset{n \to \infty}{\lim}
||g||_{T_n} =||g||_T$.
Note that $cv(F_N)\subseteq \overline{cv}(F_N)$.
It is known that $\overline{cv}(F_N)$ is precisely the closure of
$cv(F_N)$ (with respect to either of the above topologies).

\smallskip

There is a natural
\emph{continuous action} of $Out(F_N)$ on $\overline{cv}(F_N)$
that preserves $cv(F_N)$, and which can be written from the left as
well as from the right, using the convention $\phi T=T\phi^{-1}$ for $T\in \overline{cv}(F_N)$ and $\phi\in Out(F_N)$.
At the translation-length-function level this action
can be defined as follows.  For $T\in \overline{cv}(F_N)$ and
$\widehat\varphi\in Aut(F_N)$ with image $\varphi \in Out(F_{N})$ we have
$$||g||_{T \varphi}=||g||_{\varphi^{-1} T}=||\widehat\varphi(g)||_T$$
for any $g\in F_N$.

\begin{defn}[Projectivized Outer space and its compactification]$ $

\noindent (1)
For $N\ge 2$ one defines $CV(F_{N}) = cv(F_N)/\sim\,$, where $T_1 \sim T_2$
whenever there exists an $F_N$-equivariant homothety between $T_1$ and
$T_2$.  Thus $T_1\sim T_2$ in $cv(F_N)$ if and only if there is a
constant $c>0$
such that $||g||_{T_1}=c \cdot ||g||_{T_2}$ for all $g \in F_{N}$.
The space $CV(F_N)$, first introduced by M.~Culler and K.~Vogtmann \cite{CV},
is called the \emph{projectivized Outer space} or simply
\emph{Outer space}.

\smallskip
\noindent
(2)
Similarly, define $\overline {CV}(F_N)
= \overline{cv}(F_N)/\sim\,$ where
$\sim$ is again the above
homothety relation. For $T\in \overline{cv}(F_N)$
denote by $[T]$ the $\sim$-equivalence class of $T$.

\smallskip
\noindent
(c) The spaces $CV(F_N)$
and $\overline
{CV}(F_N)$ inherit the quotient topology from
$cv(F_N)$ and $\overline{cv}(F_N)$.
This makes the inclusion
$CV(F_N)\subseteq \overline {CV}(F_N)$ into a topological embedding
with dense image.  Moreover, the space $\overline {CV}(F_N)$ is
compact and thus provides a natural compactification of $CV(F_N)$.  We
also denote $\partial CV(F_N)=\overline {CV}(F_N) \smallsetminus
CV(F_N)$ and call $\partial CV(F_N)$ the \emph{Thurston boundary} of $CV(F_N)$.
\end{defn}

The natural action of $Out(F_N)$ on $\overline{cv}(F_N)$ factors
through to the action of $Out(F_N)$ by homeomorphisms on $\overline
{CV}(F_N)$.  Namely, for
$\varphi\in Out(F_N)$ and $T\in
\overline{cv}(F_N)$ we have $\varphi [T]=[\varphi T]$.
This action of $Out(F_N)$ on $\overline {CV}(F_N)$ leaves $CV(F_N)$ invariant, so
that $Out(F_N)$ acts on $CV(F_N)$ as well.

\begin{rem}\label{rem:j}
There is a standard $Out(F_N)$-equivariant topological embedding $j:CV(F_N)\to
cv(F_N)$ that gives the identity on $CV(F_{N})$ when composed
with the projection map $cv(F_N)\to
cv(F_N)/\sim\,\,=CV(F_N)$.
Namely, $j([T])=T'$, where $T'\sim T$ and the
quotient graph $T'/F_N$ has volume $1$.  One can
alternatively think about elements of $cv(F_N)$ as \emph{marked metric graph structures} on
$F_N$, as explained in more detail in Remark~\ref{rem:mgs} below.
\end{rem}

\section{Bounded back-tracking}

As before let $F_N$ be a free group of finite rank $N\ge 2$, and
let $A$ be a free basis of $F_N$.
We denote by $T_{A}$ the Cayley graph (which, of course, is a tree !) of
$F_N$ with respect to $A$.

Let $T$ be an
$\mathbb R$-tree with an isometric action of $F_N$,
and consider a point $p\in T$.
There is a unique $F_N$-equivariant map
$i_p: T_{A}\to T$ which is linear on each edge of $T_{A}$, and
which satisfies $i_p(1)=p$.

\begin{defn}[Bounded Back-Tracking constant]
The \emph{bounded back-tracking constant}
with respect to
$A$, $T$ and $p$, denoted
$BBT_{p, A}(T)$, is the infimum of all
constants $C>0$
such that for any $Q,R\in T_{A}$, the image $i_p([Q,R])$ of
$[Q,R]\subseteq T_{A}$ is contained in the $C$-neighborhood of
$[i_p(Q),i_p(R)]$.
\end{defn}

An useful result of \cite{GJLL} states:

\begin{prop}\label{prop:BBT}
Let $F_N$ be a finitely generated non-abelian free group with a minimal very
small isometric action on an $\mathbb R$-tree $T$.  Let $A$ be
a free basis of $F_N$ and let $p\in T$.

Then we have:
$$BBT_{p, A}(T)\le
\sum_{a\in A} d_T(p,ap).$$
In particular, $BBT_{p, A}(T)<\infty$.
\end{prop}

The following is an easy corollary of the definitions (see
Lemma~3.1(b) of \cite{GJLL} or Lemma~3.1 of \cite{CHL2}):

\begin{lem}\label{lem:LL}
Let $F_N$ be a finitely generated non-abelian free group with a minimal very
small isometric action on an $\mathbb R$-tree $T$.  Let $A$ be
a free basis of $F_N$ and let $p\in T$.

Suppose $BBT_{p, A}(T)<C$. Then the
following hold:
\begin{enumerate}
\item Let $w\in F(A)$ be cyclically reduced.  Then
  \[
\big| ||w||_T-d_T(p,wp) \big|\le 2C.
\]
\item Let $u=u_1\dots u_m$ be a freely reduced product of freely
reduced words in $F=F(A)$, where $m\ge 1$.  Then we have \[\left|
d_T(p,up)-\sum_{i=1}^m d_T(p,u_ip) \right|\le 2mC.\]

\item Suppose $u,u_1,\dots, u_m$ are as in (2) and that, in addition, $u$ is cyclically reduced in $F(A)$.
Then
\[
\left|
||u||_T-\sum_{i=1}^m d_T(p,u_ip) \right|\le 2mC+2C\le 4mC.
\]

\item Suppose $u,u_1,\dots, u_m$ are as in (2) and that, in addition, $u,u_1,\dots, u_m$ are cyclically reduced in $F(A)$.
Then
\[
\left| ||u||_T-\sum_{i=1}^m ||u_i||_T \right|\le 6mC.
\]
\end{enumerate}
\end{lem}

\section{Uniform approximation of $\R$-trees}

Let $A$ be a free basis of $F_N$.  Recall that  $T_A$ is the Cayley
tree of $F_N$ with respect to $A$, where all edges of $T_A$ have
length $1$.  Thus $T_A\in cv(F_N)$.
For $g\in F_N$ we denote by $|g|_A$ the freely reduced length of $g$ with respect to $A$, and we denote by $||g||_A$ the
cyclically reduced length of $g$ with respect to $A$. Thus $||g||_A=||g||_{T_A}$.

\smallskip

The following statement is a key  ingredient in the proof of the
continuity of our geometric intersection number.  We believe that it
will also turn out to be useful  in other circumstances.

\begin{prop}[Uniform Scaling Approximation]\label{prop:lgl}
Let $T\in \overline{cv}(F_N)$, let $A$ be a free basis of $F_N$ and
let $\epsilon>0$.  Then there exists a neighborhood $U_\epsilon$ of
$T$ in $\overline{cv}(F_N)$, such that for every $w\in F_N$ and every
$T_1,T_2\in U_\epsilon$ we have:
\[
\big| ||w||_{T_1}- ||w||_{T_2}\big|\le \epsilon||w||_A.\tag{$\dag$}
\]
\end{prop}
\begin{proof}

Choose $p\in T$.  Let $C>0$ be such that $d_T(p,ap)< C/N$ for every
$a\in A$, so that by Proposition~\ref{prop:BBT} we have
$BBT_{p,A}(T)<C$.  It suffices to prove the proposition for all
sufficiently small $\epsilon$, and we will assume that $\epsilon>0$
satisfies $N\epsilon\le C$.

Choose an integer $M>1$ so that $16C/M<\epsilon/2$.  Let
$0<\epsilon_1<\epsilon$ be such that $\frac{2\epsilon_1}{M}\le
\epsilon/2$.

Using the equivariant Gromov-Hausdorff-Paulin topology on
$\overline{cv}(F_N)$ it follows that there exists a
neighborhood $U_\epsilon$ of $T$ in $\overline{cv}(F_N)$ such
that for every $T'\in U_\epsilon$
the following holds:
There is some $p'\in T'$
such that for every $g\in F_N$ with $|g|_A\le M$
we have
\[
\big| d_T(p,gp)-d_{T'}(p',gp') \big|\le \epsilon_1 \, .\tag{$\ast$}
\]
Hence $BBT_{p', A}(T')\le \sum_{a\in A} d_{T'}(p',ap')<
C+N\epsilon_1\le 2C$.  We
will now verify
that the neighborhood $U_\epsilon$
satisfies the requirements of the proposition.

Let $T_1,T_2\in U_\epsilon$ be arbitrary, and let $p_1\in T_1, p_2\in
T_2$ be chosen as above.  Let $w\in F(A)$ be a non-trivial cyclically
reduced word such that $||w||_A$ is divisible by $M$.  Put
$m=||w||_A/M$.  Thus $m\ge 1$ is an integer.  Write $w$ as a freely
reduced product $w=u_1\dots u_m$ in $F(A)$, where $|u_i|_A=M$ for all
$i=1,\dots, m$.

Then, by the properties of the BBT-constant listed in
Lemma~\ref{lem:LL} (specifically, by part (3) of Lemma~\ref{lem:LL}), we have for $j=1,2\,$:
\[
\big| ||w||_{T_j}-\sum_{i=1}^m d_{T_j}(p_j,u_ip_j) \big|\le 8Cm
\]
Also, ($\ast$) implies that for $j=1,2$ the inequality
\[
\big| \sum_{i=1}^m d_{T}(p,u_ip)-\sum_{i=1}^m d_{T_j}(p_j,u_ip_j)
\big|\le m\epsilon_1
\]
holds.
This implies:
\[
\big| ||w||_{T_1}- ||w||_{T_2}\big|\le
16Cm+2m\epsilon_1=\frac{16C+2\epsilon_1}{M}||w||_A\le \epsilon
||w||_A
\]

Thus we have established that $(\dag)$ holds for every $w\in F_N$ with
$||w||_A$ divisible by $M$.

For the general case let $w\in F(A)$ be an arbitrary nontrivial cyclically reduced word.  Since $||w^M||_A=M||w||_A$ is divisible by $M$, we already know that $(\dag)$ holds for $w^M$. By dividing by $M$ both sides of the inequality $(\dag)$ for $w^M$, we conclude that $(\dag)$ holds for $w$ in view of Remark~\ref{powers}.
\end{proof}

\section{Geodesic currents}\label{sect:currents}

Let $\partial F_N$ be the hyperbolic boundary of $F_N$ (see \cite{GH}
for background information about word-hyperbolic groups).  We denote
\[
\partial^2F_N=\{(\xi_1,\xi_2): \xi_1,\xi_2\in \partial F_N, \text{ and
} \xi_1\ne \xi_2\}.
\]
Also denote by $\sigma: \partial^2 F_N\to \partial^2 F_N$ the ``flip"
map defined as $\sigma:(\xi_1,\xi_2)\mapsto (\xi_2,\xi_1)$ for
$(\xi_1,\xi_2)\in \partial^2F_N$.

\begin{defn}[Simplicial charts]
A \emph{simplicial chart} on $F_N$ is an isomorphism $\alpha:
F_N\to\pi_1(\Gamma,x)$, where $\Gamma$ is a finite connected graph
without valence-one vertices, and where $x$ is a vertex of $\Gamma$.
\end{defn}

A simplicial chart $\alpha$  on $F_N$ defines an
$F_N$-equivariant quasi-isometry between $F_N$ (with any word metric)
and the universal covering
$\widetilde \Gamma$, equipped with the simplicial metric (i.e. every edge has length $1$).
Correspondingly, we get canonical $F_N$-equivariant homeomorphisms $\partial  \alpha:
\partial F_N\to \partial \widetilde \Gamma$ and $\partial ^{2} \alpha:
\partial^2 F_N\to \partial^2 \widetilde \Gamma$, that do not depend on
the choice of a word metric for $F_N$.  If $\alpha$ is fixed, we will
usually use these homeomorphisms to identify $\partial F_N$ with
$\partial \widetilde \Gamma$ and $\partial^2 F_N$ with $\partial^2
\widetilde \Gamma$ without additional comment.

\begin{rem}\label{rem:mgs}$ $

(a) Combinatorially, we adopt Serre's convention regarding graphs. Thus every edge $e\in E\Gamma$ comes equipped with the \emph{inverse edge} $e^{-1}$, such that $e\ne e^{-1}$ and $(e^{-1})^{-1}=e$. Moreover, for every $e\in E\Gamma$, the initial vertex of $e$ is the terminal vertex of $e^{-1}$ and the terminal vertex of $e$ is the initial vertex of $e^{-1}$. An \emph{orientation} on $\Gamma$ is a partition $E\Gamma=E^+\Gamma\sqcup E^{-}\Gamma$ such that for every $e\in E\Gamma$ one of the edges $e,e^{-1}$ belongs to $E^+\Gamma$ and the other belongs to $E^-\Gamma$.

(b) Any simplicial chart $\alpha: F_N\to\pi_1(\Gamma,x)$ defines a
finite-dimensional open cell in $cv(F_N)$ and a finite-dimensional
open simplex in $CV(F_N)$.
More precisely, let $L$ be a \emph{metric graph
structure} on $\Gamma$, that is, a map $L:E\Gamma\to R_{>0}$ such that
$L(e)=L(e^{-1})$ for every edge $e\in E\Gamma$.  Then we can pull-back $L$
to $\widetilde \Gamma$ by giving every edge in $\widetilde \Gamma$ the
same length as that of its projection in $\Gamma$.  Let $d_L$ be the
resulting metric on $\widetilde \Gamma$, which makes $\widetilde
\Gamma$ into an $\mathbb R$-tree. The
action of $F_N$ on this tree, defined via
$\alpha$, is a deck transformation action and thus
minimal, free and discrete.  Hence this
action defines a point in $cv(F_N)$.  Varying the metric structure $L$
on $\Gamma$ produces an open cone $\Delta(\alpha)\subseteq cv(F_N)$ in
$cv(F_N)$, which is homeomorphic to the positive open cone in $\mathbb R^m$.
Here $m$ is
the number of topological edges of $\Gamma$, that is,
$m=\frac{1}{2}\#E\Gamma$.  Thus we can think of a simplicial chart
$\alpha: F_N\to\pi_1(\Gamma,x)$ as defining a local ``coordinate patch"
on $cv(F_N)$.

\smallskip
\noindent
(c) If we require the sum of the lengths of all the topological edges of
$\Gamma$ to be equal to $1$, we get a subset $\Delta'(\alpha)$ of
$cv(F_N)$
that is homeomorphic to an open simplex of dimension $m-1$.  This
subset $\Delta'(\alpha)$ belongs to the subset $j(CV(F_{N}))$ defined in Remark \ref{rem:j}, and hence
projects homeomorphically to its image in $CV(F_N)$.

\smallskip
\noindent
(d)
Moreover, the union of open cones $\Delta(\alpha)$ over all simplicial
charts $\alpha$ is equal to $cv(F_N)$,
and this union is a disjoint union.
Additionally, every point of $cv(F_N)$
belongs to only a finite number of closures $\overline{\Delta}(\alpha)$
of such open cones.  Similarly, the
copy $j(CV(F_N))$ of $CV(F_N)$ in $cv(F_N)$ is the
disjoint union of the open simplices $\Delta'(\alpha)$ over all
simplicial charts $\alpha$, and the closures of these open simplices in $cv(F_N)$
form a locally finite cover of $j(CV(F_N))$.

\end{rem}

\begin{defn}[Cylinders]
Let $\alpha: F_N\to\pi_1(\Gamma,x)$ be a simplicial chart on $F_N$.
For a non-trivial reduced edge-path $\gamma$ in $\widetilde \Gamma$
denote by $Cyl_{\widetilde \Gamma}(\gamma)$ the set of all
$(\xi_1,\xi_2)\in \partial^2 F_N$ such that the bi-infinite geodesic
from $\tilde \alpha(\xi_1)$ to $\tilde \alpha(\xi_2)$ contains
$\gamma$ as a subpath.

We call $Cyl_{\widetilde \Gamma}(\gamma)\subseteq \partial^2 F_N$ the
\emph{two-sided cylinder corresponding to $\gamma$}.
\end{defn}

It is easy to see that $Cyl_{\widetilde \Gamma}(\gamma)\subseteq
\partial^2 F_N$ is both compact and open.  Moreover, the collection of
all such cylinders, where $\gamma$ varies over all non-trivial reduced
edge-paths in $\widetilde \Gamma$, forms a basis of open sets in
$\partial^2 F_N$.

\begin{defn}[Geodesic currents]\label{defn:current}
A \emph{geodesic current}
(or simply {\em current})
on $F_N$ is a positive Radon measure on
$\partial^2 F_N$ which
is $F_N$-invariant and $\sigma$-invariant.  The
set of all geodesic currents on $F_N$ is denoted by $Curr(F_N)$.  The
set $Curr(F_N)$ is endowed with the weak-* topology.
This makes $Curr(F_N)$ into a locally compact space.
\end{defn}

Specifically, let $\alpha: F_N\to\pi_1(\Gamma,x)$ be a simplicial
chart on $F_N$.  Let $\mu_n,\mu\in Curr(F_N)$.  It is not hard to
show~\cite{Ka2} that $\underset{n \to \infty}{\lim}\mu_n=\mu$ in $Curr(F_N)$ if
and only if for every non-trivial reduced edge-path $\gamma$ in
$\widetilde\Gamma$ we have
\[
\underset{n \to \infty}{\lim} \mu_n(Cyl_{\widetilde
\Gamma}(\gamma))=\mu(Cyl_{\widetilde \Gamma}(\gamma)).
\]

Let $\mu\in Curr(F_N)$ and let $v$ be a non-trivial reduced edge-path
in $\Gamma$.  Denote
\[
\langle v, \mu\rangle_{\alpha}:=\mu(Cyl_{\widetilde \Gamma}(\gamma)),
\]
where $\gamma$ is any edge-path in $\widetilde \Gamma$ that is
labelled by $v$.  Since $\mu$ is $F_N$-invariant, this definition does
not depend on the choice of
the lift $\gamma$ of $v$.

\medskip

There is a natural continuous left-action of $Aut(F_N)$ on $Curr(F_N)$
by linear transformations.  Namely, let $\psi\in Aut(F_N)$.  Then
$\psi$ is a quasi-isometry of $F_N$ and hence $\psi$ induces a
homeomorphism $\partial \psi$
of $\partial F_N$ and hence a homeomorphism $\partial^{2}
\psi: \partial^2F_N\to\partial^2F_N$.  Then for $\mu\in Curr(F_N)$ and
$S\subseteq \partial^2F_N$ put
\[
(\psi \mu)(S):=\mu({\partial^{2}\psi}^{-1} S).
\]
It is not hard to check~\cite{Ka2} that $\psi\mu$ is indeed a geodesic
current.  Moreover, the group of inner automorphisms $Inn(F_N)$ is
contained in the kernel of the action of $Aut(F_N)$ on $Curr(F_N)$.
Therefore this action factors through to a continuous action of
$Out(F_N)$ on $Curr(F_N)$.

\begin{notation} \label{ginfty} $ $

\noindent (1) For any $g\in F_N
\smallsetminus \{1\}$ denote by
$g^{\infty}=\underset{n \to \infty}{\lim} g^n$ and $g^{-\infty}=\underset{n \to \infty}{\lim}
g^n$ the two distinct limit points in $\partial F_{N}$.
Hence one obtains
$(g^{-\infty},g^{\infty})\in \partial^2F_N$.

\smallskip
\noindent
(2)
For any $g\in F_N$ we will denote by $[g]_{F_N}$ or just by
$[g]$ the conjugacy class of $g$ in $F_N$.
\end{notation}

\begin{defn}[Counting and Rational Currents]\label{defn:rational}
(1)
Let $g\in F_N$ be a non-trivial element that is not a proper power in
$F_N$.  Set
\[
\eta_g=\sum_{h\in [g]}
\left(\delta_{(h^{-\infty},h^{\infty})}+
\delta_{(h^{\infty},h^{-\infty})}\right) \, ,
\]
where $\delta_{(h^{-\infty},h^{\infty})}$ denotes as usually the
atomic Dirac (or ``counting") measure concentrated at the point
${(h^{-\infty},h^{\infty})}$.

Let $\mathcal R(g)$ be the collection of all $F_N$-translates of
$(g^{-\infty},g^{\infty})$ and $(g^{\infty},g^{-\infty})$ in
$\partial^2 F_N$.
This gives
\[
\eta_g=\sum_{(x,y)\in \mathcal R(g)} \delta_{(x,y)},
\]
and hence $\eta_g$ is $F_N$-invariant and flip-invariant, that is
$\eta_g\in Curr(F_N)$.

\smallskip
\noindent
(2)
Let $g\in F_N \smallsetminus \{1\}$ be
arbitrary.
Write $g=f^m$
where $m\ge 1$ and $f\in F_N$ is not a proper power,
and define:
$$\eta_g:=m \cdot \eta_f.$$
We call $\eta_g\in Curr(F_N)$ the
\emph{counting current} given by $g$.  Non-negative scalar multiples of counting currents are called
\emph{rational currents}.
\end{defn}

It is easy to see that if $[g]=[h]$ then $\eta_g=\eta_h$ and
$\eta_g=\eta_{g^{-1}}$.  Thus $\eta_g$ depends only on the conjugacy
class of $g$ and we will sometimes denote $\eta_{[g]}:=\eta_g$.
Moreover, it is not hard to check~\cite{Ka2} that for $\varphi\in
Out(F_N)$ and $g\in F_N\smallsetminus \{1\}$ we have $\varphi\, \eta_{[g]}=
\eta_{\varphi[g]}$.
One can also give a more explicit combinatorial description of
the counting current $\eta_g$ in terms of counting the numbers of
occurrences of freely reduced words in a ``cyclic word" $w$
representing $g$ (with respect to some fixed free basis of $F_N$).  We
refer the reader to \cite{Ka2} for details.

\begin{prop}\cite{Ka1,Ka2}
The set of all rational currents is dense in the space $Curr(F_N)$.
\end{prop}

\begin{defn}[Projectivized space of geodesic currents]
Let $N\ge 2$.  We define
$$\mathbb PCurr(F_N)=Curr(F_N) \smallsetminus\{0\}\,/\sim$$
where two currents $\mu_1,\mu_2\in Curr(F_N)\smallsetminus\{0\}$
satisfy $\mu_1\sim\mu_2$ if there is some constant
$c>0$ such that $\mu_2=c\, \mu_1$.
For a nonzero current $\mu\in Curr(F_N)$ denote by $[\mu]$ the projective
class of $\mu$, that is, the $\sim$-equivalence class of $\mu$.

The quotient set $\mathbb PCurr(F_N)$ inherits the quotient topology from
$Curr(F_N)$ as well as a continuous action of $Out(F_N)$.  The space
$\mathbb PCurr(F_N)$ is called the \emph{projectivized space of
geodesic currents} (or simply {\em space of projectivized currents})
on $F_N$.
\end{defn}

It is known~\cite{Ka1,Ka2} that $\mathbb PCurr(F_N)$ is compact.

\section{The intersection form}

In this section we will prove the main result of this paper, whose
slightly simplified version was stated in the Introduction as Theorem~A.
We state our result now in its full strength,
using the terminology introduced in the previous sections.

\subsection{Statement of the main result}

\begin{thm}\label{thm:int-form}
Let $N\ge 2$ be an integer.
There exists a unique \emph{geometric intersection form}
$$\langle \ , \ \rangle: \overline{cv}(F_N)\times Curr(F_N)\to
\mathbb R_{\ge 0}$$
with the following properties.

\begin{enumerate}
\item The function $\langle\ , \ \rangle$
is continuous.

\item The function
$\langle\ , \ \rangle$ is
$R_{\ge 0}$-homogeneous in the first argument. Namely,
for any $T\in \overline{cv}(F_N)$, $\mu \in Curr(F_N)$ and
$\lambda\ge 0$ we have:
\[
\langle \lambda T, \mu\rangle=\lambda\langle T,\mu\rangle
\]

\item
The function
$\langle\ , \ \rangle$
 is $R_{\ge 0}$-linear in
the second argument. Namely,
for any $T\in \overline{cv}(F_N)$, $\mu_1,\mu_2\in Curr(F_N)$
$\lambda_1,\lambda_2\ge 0$ we have:
\[
\langle T, \lambda_1 \mu_1+\lambda_2\mu_2\rangle=\lambda_1\langle
T,\mu_1\rangle+\lambda_2\langle T,\mu_2\rangle
\]

\item The function $\langle\ , \ \rangle$
is $Out(F_N)$-invariant: for any $T\in \overline{cv}(F_N)$, $\mu\in Curr(F_N)$ and
$\varphi\in Out(F_N)$ we have:
\[
\langle \varphi T, \varphi \mu\rangle=\langle T,\mu\rangle
\]

\item
For any $T \in \overline{cv}(F_N)$ and any $g \in F_{N}$, with
associated counting current $\eta_{g} \in Curr(F_{N})$, we have:
 \[
\langle T,\eta_{g}\rangle\,\, = \, \,
||g||_T
\]

\item Let $\alpha:F\to\pi_1(\Gamma,x)$ be a simplicial chart on $F$ and let $L:E\Gamma\to \mathbb R_{>0}$ be a metric graph structure on $\Gamma$ and let $T\in cv(F)$ be the tree corresponding to the pull-back of $L$ to $\widetilde\Gamma$, with the action of $F_N$ on $T$ via $\alpha$.
Then for any $\mu\in Curr(F_N)$ we have:
\[
\langle \widetilde \Gamma, \mu\rangle
=\sum_{e\in E^+\Gamma} L(e) \langle e,
\mu\rangle_{\alpha},
\]
where $E\Gamma=E^+\Gamma\sqcup E^{-}\Gamma$ is an orientation on $\Gamma$.
\end{enumerate}
\end{thm}

\begin{rem}\label{rem:easy}$\, $

(a) Note that conditions (1), (3) and (5) already imply that if such an intersection form exists, then it is unique.
Indeed, recall that the set of rational currents is dense in $Curr(F)$. Thus if $\mu\in Curr(F)$ then there exists a sequence of rational currents $\lambda_i \eta_{g_i}$ such that $\mu=\lim_{i\to\infty} \lambda_i\eta_{g_i}$. Hence the continuity and linearity of the intersection form imply that
\[
\langle T, \mu\rangle=\lim_{i\to\infty} \lambda_i ||g_i||_T.
\]

Thus Theorem~\ref{thm:int-form} implicitly implies that the above limit exists and does not depend on the choice of the sequence of rational currents converging to $\mu$.

(b) For the case of $cv(F_N)$ the statement of Theorem~\ref{thm:int-form} was already obtained in \cite{Ka2,L}, where we constructed the intersection form with the required properties on $cv(F_N)\times Curr(F_N)$. The main difficulty that had to be overcome in the present paper is to prove that that intersection form admits a continuous ``boundary" extension to a continuous map $\overline{cv}(F_N)\times Curr(F_N)\to \mathbb R$.

\smallskip
\noindent

(c) Note that the
$Out(F_{N})$-equivariance equality given in
part (4) of Theorem \ref{thm:int-form}
is equivalent to the formula $$\langle T \varphi, \mu \rangle =
\langle T , \varphi \mu \rangle \, ,$$
as follows directly from the fact that the left side of this equation
is equal to $\langle \varphi^{-1} T , \mu \rangle$ (see the definition of
the $Out(F_{N})$-action in Section~\ref{sect:cv}).
\end{rem}

\subsection{The case of $cv(F_N)$}$ $

\smallskip

In \cite{Ka2,L} we established the statement of Theorem~\ref{thm:int-form} for $cv(F_N)$:

\begin{propdfn}[Intersection Form on $cv(F_N)$]\label{pd:int-form}

Let $N\ge 2$. There exists a unique map
\[
\langle\ , \ \rangle: cv(F_N)\times Curr(F_N)\to \mathbb R_{\ge 0}
\]
satisfying conditions (1)-(6) of Theorem~\ref{thm:int-form} for arbitrary $T\in cv(F)$.

For $T\in cv(F_N)$ and $\mu\in Curr(F_N)$ we call $\langle T,\mu\rangle$ the
\emph{geometric intersection number of $T$ and $\mu$}.
\end{propdfn}

Note, that, as we have seen in Remark~\ref{rem:easy}, if $T\in cv(F_N)$ and $\mu\in Curr(F_N)$ is represented as
$\mu=\lim_{i\to\infty} \lambda_i\eta_{g_i}$ for some $g_i\in F_N$ and $\lambda_i\ge 0$ then
\[
\langle T, \mu\rangle=\lim_{i\to\infty} \lambda_i ||g_i||_T.
\]

\subsection{Continuous extension of the intersection form to $\overline{cv}(F_N)$.}$ $

\smallskip

The main tool to prove the existence of a continuous extension of the intersection form to $\overline{cv}(F_N)$ will be
Proposition~\ref{prop:lgl}.
We first prove:
\begin{prop}\label{prop:limit}
Let $T\in \overline{cv}(F_N)$ and let $\mu\in Curr(F_N)$ be such that
$\mu=\lim_{i\to\infty}\lambda_i \eta_{g_i}$ for some $g_i\in F_N$ and
$\lambda_i\ge 0$.  Then the limit
\[
\lim_{i\to\infty}\lambda_i ||g_i||_T
\]
exists and does not depend on the choice of the sequence $\lambda_i
\eta_{g_i}$ of the rational currents that converges to $\mu$.
\end{prop}

\begin{proof}
Fix a free basis $A$ of $F_N$.
Let $g_i\in F_N$ and $\lambda_i\ge 0$ be such that
$\mu=\lim_{i\to\infty}\lambda_i \eta_{g_i}$.  We claim that
$\lambda_i||g_i||_T$ is a Cauchy sequence of real numbers and hence
has a finite limit.

Let $\epsilon>0$ be arbitrary.
Choose $0<\epsilon_1<\epsilon$ such that $2\epsilon_1(\langle T_A,
\mu\rangle+\epsilon_1)+\epsilon_{1}
\le\epsilon$.  Note that we allow for the
possibility that $\mu=0$.

Let $U_{\epsilon_1}$ be the neighborhood of $T$ provided by
Proposition~\ref{prop:lgl}.  Choose a tree $T'\in U_{\epsilon_1}$ such
that $T'\in cv(F_N)$.  Then $\big| ||g_i||_T-||g_i||_{T'} \big|\le
\epsilon_1 ||g_i||_A$ and hence
\[
\big| \lambda_i||g_i||_T-\lambda_i||g_i||_{T'} \big|\le
\epsilon_1\lambda_i ||g_i||_A.
\]

Recall that $\lim_{i\to\infty} \lambda_i||g_i||_{T'}=\langle
T',\mu\rangle$ and $\lim_{i\to\infty} \lambda_i||g_i||_{A}=\langle
T_A,\mu\rangle$ since $T',T_A\in cv(F_N)$.

Thus there is $i_0\ge 1$ such that for every $i\ge i_0$ we have
$|\lambda_i||g_i||_{T'}-\langle T',\mu\rangle|\le \epsilon_1$ and
$\lambda_i||g_i||_{A}\le \langle T_A,\mu\rangle+\epsilon_1$.

Thus for every $i\ge i_0$ we have
\[
\big| \lambda_i||g_i||_T-\langle T',\mu\rangle \big|\le
\epsilon_1(\langle T_A, \mu\rangle+\epsilon_1)
+\epsilon_1.
\]
This implies that the sequence $\lambda_i||g_i||_T$ is bounded and,
moreover, for any $i,j\ge i_0$
\[
\big| \lambda_i||g_i||_T-\lambda_j||g_j||_T\big|\le
2(\epsilon_1(\langle T_A, \mu\rangle+\epsilon_1) +\epsilon_1)\le \epsilon.
\]
Since $\epsilon>0$ was arbitrary, this shows that $\lambda_i||g_i||_T$
is a Cauchy sequence of real numbers and hence has a finite limit in
$\mathbb R$.

It is now clear that this limit does not depend on the choice of a
sequence of rational currents $\lambda_i\eta_{g_i}$ such that
$\underset{i\to\infty}{\lim} \lambda_i\eta_{g_i}=\mu$,
since one can mix any two such sequences together to produce a new sequence of rational currents
also limiting to $\mu$.

\end{proof}

Proposition~\ref{prop:limit} implies that the following notion is
well-defined:
\begin{defn}[Intersection form on $\overline{cv}(F_N)$]\label{defn:ext}
Let $T\in \overline{cv}(F_N)$ and let $\mu\in Curr(F_N)$.  Put
\[
\langle T,\mu\rangle=\lim_{i\to\infty} \lambda_i ||g_i||_T
\]
where $g_i\in F_N$ and $\lambda_i\ge 0$ are any such that
$\mu=\lim_{i\to\infty} \lambda_i \eta_{g_i}$.
\end{defn}
Note that the intersection number from Definition~\ref{defn:ext} agrees with the intersection number from
from Proposition-Definition~\ref{pd:int-form} for arbitrary $T\in cv(F_N)$ and $\mu\in Curr(F_N)$.

\begin{lem}\label{lem:use}
Let $A$ be a free basis of $F_N$.  Let $T\in \overline{cv}(F_N)$.  Let
$\epsilon>0$ and let $U_\epsilon$ be the neighborhood of $T$ in
$\overline{cv}(F_N)$ provided by Proposition~\ref{prop:lgl}.  Then for
any $T_1,T_2\in U_\epsilon$ and for any $\nu\in Curr(F_N)$ have
\[
\big| \langle T_1,\nu\rangle-\langle T_2,\nu\rangle \big|\le 2\epsilon
\langle T_A, \nu\rangle.
\]

\end{lem}

\begin{proof}

The statement is obvious if $\nu=0$ so we will assume that $\nu\ne 0$.
Hence $\langle T_A,\nu\rangle>0$ and $\langle T_0,\nu\rangle>0$.  Let
$\epsilon_1>0$ be such that $\epsilon (\langle
T_A,\nu\rangle+\epsilon_1)+2\epsilon_1\le 2\epsilon\langle
T_A,\nu\rangle$.

Let $\nu=\lim_{i\to\infty} \lambda_i\eta_{g_i}$ for some $g_i\in F_N$
and $\lambda_i\ge 0$.  Choose $i_0\ge 1$ such that for every $i\ge
i_0$
\[
\big| \langle T_j,\nu\rangle-\lambda_i||g_i||_{T_j}\big|\le
\epsilon_1, \text{ for } j=1,2
\]
and
\[
\big| \langle T_A,\nu\rangle-\lambda_i||g_i||_{A}\big|\le \epsilon_1.
\]

Then for $i\ge i_0$ we have, by Proposition~\ref{prop:lgl}:
\begin{gather*}
\big| \langle T_1,\nu\rangle-\langle T_2,\nu\rangle \big|\le \big|
\lambda_i||g_i||_{T_1}-\lambda_i||g_i||_{T_2}\big|+2\epsilon_1\le \\
\le \epsilon\lambda_i ||g_i||_A+2\epsilon_1\le \epsilon (\langle
T_A,\nu\rangle+\epsilon_1)+2\epsilon_1\le 2\epsilon\langle
T_A,\nu\rangle.
\end{gather*}
\end{proof}

\begin{proof}[Proof of Theorem \ref{thm:int-form}]

We first show that the map $\langle\, , \, \rangle: \overline{cv}(F_N)\times Curr(F_N)\to \mathbb R_{\ge 0}$, given in
Definition~\ref{defn:ext}, is continuous.

Choose a free basis $A$ of $F_N$,
and let $T\in \overline{cv}(F_N)$, $\mu\in Curr(F_N)$ and $\epsilon>0$ be
arbitrary.

Let $\epsilon_1>0$ be such that $4\epsilon_1 \langle T_A,\mu\rangle\le
\epsilon/2$.  Let $\epsilon_2>0$ be such that
$2\epsilon_1\epsilon_2+\epsilon_2\le \epsilon/2$.

Let $U_{\epsilon_1}\subseteq \overline{cv}(F_N)$ be the neighborhood
of $T$ in $\overline{cv}(F_N)$ provided by Proposition~\ref{prop:lgl}.
Choose $T_0\in U_{\epsilon_1}\cap cv(F_N)$.

Since $\langle \ , \ \rangle: {cv}(F_N)\times Curr(F_N)\to \mathbb R$
is continuous and since $T_0,T_A\in cv(F_N)$, there exists a
neighborhood $V$ of $\mu$ in $Curr(F_N)$ such that for every $\mu'\in
V$ we have
\[
\big| \langle T_0,\mu'\rangle- \langle T_0,\mu\rangle\big|\le
\epsilon_2
\]
and
\[
\big| \langle T_A,\mu'\rangle- \langle T_A,\mu\rangle\big|\le
\epsilon_2.
\]

Now let $T'\in U_{\epsilon_1}$ and $\mu'\in V$ be arbitrary.  By
Lemma~\ref{lem:use} we have
\begin{gather*}
\big| \langle T',\mu'\rangle- \langle T,\mu\rangle\big|\le\\ \big|
\langle T',\mu'\rangle- \langle T_0,\mu'\rangle\big|+\big| \langle
T_0,\mu'\rangle- \langle T_0,\mu\rangle\big|+\big| \langle
T_0,\mu\rangle- \langle T,\mu\rangle\big|\le \\
\le 2\epsilon_1 \langle T_A,\mu'\rangle+\epsilon_2+2\epsilon_1\langle
T_A,\mu\rangle\le\\
2\epsilon_1 \langle
T_A,\mu\rangle+2\epsilon_1\epsilon_2+\epsilon_2+2\epsilon_1\langle
T_A,\mu\rangle=4\epsilon_1 \langle
T_A,\mu\rangle+2\epsilon_1\epsilon_2+\epsilon_2\le \epsilon.
\end{gather*}

Since $\epsilon>0$ was arbitrary, this implies that $\langle \ , \
\rangle: \overline{cv}(F_N)\times Curr(F_N)\to \mathbb R$ is
continuous at the point $(T,\mu)$.  Since $(T,\mu)\in
\overline{cv}(F_N)\times Curr(F_N)$ was arbitrary, it follows that
$\langle \ , \ \rangle: \overline{cv}(F_N)\times Curr(F_N)\to \mathbb
R$ is continuous, as required. This establishes part (1) of Theorem \ref{thm:int-form}.

The fact that parts (1)-(5) of Theorem \ref{thm:int-form} hold now follows by continuity from
the same properties known to hold for $\langle \ , \ \rangle:
{cv}(F_N)\times Curr(F_N)\to \mathbb R$.
Part (6) of Theorem \ref{thm:int-form} only concerns $\mathbb R$-trees from $cv(F_N)$ and is thus already known (see Proposition-Definition~\ref{pd:int-form} above).
\end{proof}

\section{The intersection form and iwip automorphisms of $F_{N}$}

\begin{notation}
Note that if $T, T'\in \overline{cv}(F_N)$ and $\mu,\mu'\in Curr(F_N),
\mu\ne 0, \mu'\ne 0$ are such that $[T]=[T']$ and $[\mu]=[\mu']$ then
$\langle T,\mu\rangle=0$ if and only if $\langle T',\mu'\rangle=0$.
Therefore for $x\in \overline {CV}(F_N)$, $y\in \mathbb PCurr(F_N)$ we
will write $\langle x,y\rangle=0$ if for some (or equivalently, for any)
$T\in \overline{cv}(F_N)$, $\mu\in Curr(F_N)$ with $[T]=x$ and
$[\mu]=y$ we have $\langle T,\mu\rangle=0$.
\end{notation}

\begin{lem}\label{lem:0}
Let $[T_n],[T]\in \overline {CV}(F_N)$ and $[\mu_n]\in \mathbb
PCurr(F_N)$ be such that $\underset{n \to \infty}{\lim} [T_n]=[T]$ and
$\underset{n \to \infty}{\lim}[\mu_n]=[\mu]$, and such that $\langle
[T_n],[\mu_n]\rangle=0$ for every $n\ge 1$.  Then
$$\langle
[T],[\mu]\rangle=0.$$
\end{lem}

\begin{proof}
There exist $r_n\ge 0$ and $c_n\ge 0$ such that $T=\underset{n \to
\infty}{\lim} r_nT_n$ and $\mu=\underset{n \to \infty}{\lim}
c_n\mu_n$.  By linearity of the intersection form
we have $\langle
r_nT_n,c_n\mu_n\rangle=r_nc_n\langle T_n,\mu_n\rangle=0$.  Hence by
continuity (part (1) of Theorem~\ref{thm:int-form})
we have $\langle
T,\mu\rangle=\underset{n \to \infty}{\lim}\langle
r_nT_n,c_n\mu_n\rangle=0$, as required.
\end{proof}

\begin{defn}[IWIP]\label{defn:iwip}
As in \cite{LL}, we say that an outer automorphism
$\varphi\in Out(F_N)$ is
{\em irreducible with irreducible powers}  or an \emph{iwip}
if no conjugacy class of any non-trivial proper free factor of $F_N$ is mapped
by a positive power of $\varphi$ to itself.
\end{defn}

It is known that if such an iwip $\varphi$
is without periodic conjugacy classes, then
$\varphi$ has a ``North-South" dynamics for its induced actions on both, $\overline
{CV}(F_N)$ and $\mathbb PCurr(F_N)$:

\begin{prop}\label{prop:ns}
Let $N\ge 3$ and let $\varphi\in Out(F_N)$ be irreducible with
irreducible powers.  Then the following hold:
\begin{enumerate}
\item {\rm (Levitt-Lustig~\cite{LL})} The action of $\varphi$ on
$\overline {CV}(F_N)$ has precisely two distinct fixed points, $[T_+]$
and $[T_-]$, that both belong to $\partial CV(F_N)$.  Moreover, for
any $[T]\ne [T_-]$ in $\overline {CV}(F_N)$ we have
$\underset{n \to \infty}{\lim} \varphi^n[T]=[T_+]$.  Similarly, for any
$[T]\ne [T_+]$
in $\overline {CV}(F_N)$ we have $\underset{n \to \infty}{\lim}
\varphi^{-n}[T]=[T_-]$.

\item {\rm (Reiner Martin~\cite{Ma})} Suppose in addition that $\varphi$
has no periodic conjugacy classes in $F_N$.  Then the action of $\varphi$
on $\mathbb PCurr(F_N)$ has precisely two distinct fixed points
$[\mu_+]$ and $[\mu_-]$.  Moreover, for any $[\mu]\ne [\mu_-]$ in $\mathbb
PCurr(F_N)$
we have $\underset{n \to \infty}{\lim}
\varphi^n[\mu]=[\mu_+]$.  Similarly, for any $[\mu]\ne [\mu_+]$ in
$\mathbb PCurr(F_N)$
we have $\underset{n \to \infty}{\lim}
\varphi^{-n}[\mu]=[\mu_-]$.
\end{enumerate}
\end{prop}

\begin{conv}\label{conv:iwip}
For the remainder of this section, unless specified otherwise, let
$N\ge 3$ and let $\varphi\in Out(F_N)$ be irreducible with irreducible
powers, and without periodic conjugacy classes.  Let $[T_+], [T_{-}]\in
\partial CV(F_N)$ be the attracting and repelling fixed points for the
(left) action of $\varphi$ on $\overline {CV}(F_N)$.  Similarly, let
$[\mu_+], [\mu_{-}]\in \mathbb PCurr(F_N)$ be the attracting and
repelling fixed points for the action of $\varphi$ on $\mathbb
PCurr(F_N)$.
\end{conv}

\begin{rem}
\label{othersources}
(1) We would like to alert the reader that
the {\em forward limit tree} of $\varphi$, denoted in
\cite{CHL3} and \cite{Lu2} by $T_{\varphi}$,
is the tree $T_{-}$ (and not $T_{+}$).
This is due to the fact that in this paper $\varphi$ acts on $\R$-trees in
$\overline{CV}(F_{n})$ from the left, while \cite{CHL3} and \cite{Lu2}
in one considers the right-action (compare the discussion in Section~\ref{sect:cv}).

\smallskip
\noindent
(2) Some useful information about iwips and their induced action on Outer space
has been worked out in detail in \cite{Lu2}, \S 4 and \S 5. A summary of the
most important facts is given in \cite{CHL3}, Remark 5.5.

\smallskip
\noindent
(3) An alternative proof (relying on the main result of \cite{KL2})
for Proposition \ref{prop:non-zero} below is given by Proposition 5.6 of \cite{CHL3}.
\end{rem}

\begin{prop}\label{prop:non-zero}
Let $\varphi, T_{\pm}$ and $\mu_{\pm}$ be as in
Convention~\ref{conv:iwip}.  Then
$$\langle T_-,\mu_+\rangle\ne 0 \quad and \quad
\langle T_+,\mu_-\rangle\ne 0.$$
\end{prop}

\begin{proof}

Let $\alpha: F_N\to \pi_1(\Gamma)$ be a marked
graph structure on $F_N$,
given by a train track map on $\Gamma$ that represents $\varphi$, with a
metric structure $L$ on the edges of $\Gamma$
given by the Perron-Frobenius eigen-vector of the transition matrix (see \cite{BH92} or  Section~3 of \cite{Lu2} for a detailed exposition).
Let $T=\widetilde \Gamma \in cv(F_N)$ be the discrete $\mathbb R$-tree given by the universal
covering of $\Gamma$, provided with the metric $d_L$ given by the
lift of $L$, and with the action of $F_N$ coming from the marking
$\alpha$.

Let $\lambda>1$ be the train-track stretching constant for $\Gamma$
(i.e. the Perron-Frobenius eigen-value of the transition matrix of the
train track map).
It is known (see, for example, Remark 5.4 of \cite{Lu2})
that $\underset{n \to \infty}{\lim} \frac{1}{\lambda^n}
\varphi^{-n}T=T_-$.

Let $g\in F_N$, $g\ne 1$ be arbitrary.  Then there exists constants
$C>1$  and $n_{0} \in \mathbb N$ such
that for every $n\ge n_{0}$ we have
\[
 \frac{1}{C}\lambda^n\le ||\varphi^n(g)||_{T}\le C\lambda^n.
\]

The upper bound is derived in Section~3 of \cite{Lu2} before Remark 3.4:
The inequality becomes an equality if $g$ is represented by a legal
loop. The lower bound follows from
the fact that every path in $\Gamma$ has an iterate (under the train
track map) that is a legal composition of legal paths and INP's, see
\cite{BH92}.

Note that $||\varphi^n(g)||_{T}=||g||_{\varphi^{-n} T}$.  It was proved by
Reiner Martin~\cite{Ma} that $\underset{n \to \infty}{\lim}[\varphi^n
\eta_g]=[\mu_+]$ and, moreover, that, after possibly multiplying
$\mu_+$ by a positive scalar, we have $\underset{n \to \infty}{\lim}
\frac{1}{\lambda^n}\eta_{\varphi^n(g)}=\mu_+$.
We compute:
\begin{gather*}
\langle T_-,\mu_+\rangle
=\underset{n \to \infty}{\lim} \langle \frac{1}{\lambda^n}
\varphi^{-n}T, \frac{1}{\lambda^n}\eta_{\varphi^n(g)}\rangle \\
=\underset{n \to \infty}{\lim}\frac{1}{\lambda^{2n}}\langle \varphi^{-n}T,
\eta_{\varphi^n(g)}\rangle
=\underset{n \to \infty}{\lim}\frac{1}{\lambda^{2n}}\langle
T, \varphi^{n}\eta_{\varphi^n(g)}\rangle \\
=\underset{n \to \infty}{\lim}\frac{1}{\lambda^{2n}}\langle T,
\eta_{\varphi^{2n}(g)}\rangle
=\underset{n \to \infty}{\lim}\frac{1}{\lambda^{2n}}
||\varphi^{2n}(g)||_T \\
\ge
\underset{n \to \infty}{\lim}\frac{1}{\lambda^{2n}} C\lambda^{2n}
=C>0
\end{gather*}

Replacing $\varphi$ by $\varphi^{-1}$
we conclude that $\langle T_+,\mu_-\rangle>0$ as well.

\end{proof}

\begin{prop}\label{prop:key}
Let $[T_n]\in \overline {CV}(F_N)$ and
$[\mu_n]\in \mathbb PCurr(F_N)$ be sequences such that
$$\langle [T_n],[\mu_n]\rangle=0$$
for every $n\ge 1$.  Then we have:
$$\displaystyle\underset{n \to \infty}{\lim} [T_n]=[T_+]
\,\,\, \Longleftrightarrow \,\,\,
\displaystyle\underset{n \to \infty}{\lim} [\mu_n]=[\mu_+].$$
\end{prop}

\begin{proof}
Let $\underset{n \to \infty}{\lim} [T_n]=[T_+]$.  Suppose that
$\underset{n \to \infty}{\lim} [\mu_n]\ne [\mu_+]$.

Since $\mathbb PCurr(F_N)$ is compact,
after passing to a subsequence we may assume that $\underset{n \to
\infty}{\lim} [\mu_n]=[\mu]$ for some $[\mu]\ne [\mu_+]$ in $\mathbb
PCurr(F_N)$.  Note that by Lemma~\ref{lem:0} we have $\langle
[T_+],[\mu]\rangle=0$.  Since $[\mu]\ne [\mu_+]$,
part (2) of Proposition \ref{prop:ns} implies that
$\underset{n \to \infty}{\lim} \varphi^{-n}[\mu]=[\mu_{-}]$.  Note that $[T_+]$ is
fixed by $\varphi^{-1}$, and that we have
$$\langle [T_+],
\varphi^{-n}[\mu]\rangle=\langle
\varphi^{-n}[T_+], \varphi^{-n}[\mu]\rangle=\langle
[T_+],[\mu]\rangle=0$$
for every $n\ge 1$.  Hence Lemma~\ref{lem:0} implies
$\langle [T_+],[\mu_{-}]\rangle=0$.
This is to say that $\langle T_+,\mu_{-}\rangle=0$, yielding a contradiction with
Proposition~\ref{prop:non-zero}.  Thus $\underset{n \to \infty}{\lim}
[\mu_n]=[\mu_+]$.

The argument for the other direction is completely symmetric.
\end{proof}

\section{Curve complex analogues for free groups}

\subsection{The bipartite intersection graph}

\begin{defn}[Intersection graph]
Let $\mathcal I=\mathcal I(F_N)$ be a bipartite graph defined as
follows.  The vertex set of $\mathcal I$ is $V\mathcal I=\overline
{CV}(F_N)\cup \mathbb PCurr(F_N)$.  Two vertices $[T]\in \overline
{CV}(F_N)$ and $[\mu]\in \mathbb PCurr(F_N)$ are connected by an edge
in $\mathcal I$ if and only if $\langle [T],[\mu]\rangle=0$.
\end{defn}

Since the intersection form is $Out(F_N)$-invariant, the
graph $\mathcal I(F_N)$ comes equipped with a natural action of $Out(F_N)$
by graph automorphisms. 

It is not hard to show that for $N\ge 3$ there is a single connected
component $\mathcal I_0(F_N)$ of $\mathcal I(F_N)$ containing all
$[T]$ for the Bass-Serre trees $T$ corresponding to all nontrivial
free product decompositions of $F_N$. Moreover, the same connected
component $\mathcal I_0(F_N)$  also contains $\eta_a$ for all
nontrivial $a\in F_N$ that belong to some proper free factors of
$F_N$. It is also not hard to show that $\mathcal I_0(F_N)$ is $Out(F_N)$-invariant.

Note also that there are many connected components in this graph.
Indeed, every vertex $[T]\in {CV}(F_N)$ is an isolated point, and it
follows from \cite{CHL3} that many pairs $([T], [\mu])$ form a single edge connected component, in particular all pairs
$([T_{+}], [\mu_{+}])$ as in Convention \ref{conv:iwip}.

\begin{prop}\label{prop:inf}
Let $[T_n],[T]\in \overline {CV}(F_N)$ be such that $[T]\ne [T_+]$ and
that $\underset{n \to \infty}{\lim}[T_n]=[T_+]$, for $[T_{+}]$ as in
Convention \ref{conv:iwip}.
Then in the graph $\mathcal I$
we have:
$$\underset{n \to \infty}{\lim} d_{\mathcal I}([T_n],[T])=\infty.$$
\end{prop}

\begin{proof}
Suppose that the statement of the lemma fails.  Then there exists a
sequence $[T_n]\in \overline{CV}(F_N)$ with
$\underset{n \to \infty}{\lim}[T_n]=[T_+]$, such that $\underset{n\ge
1}{\max}\, \, d_\mathcal
I([T_n],[T])<\infty$.  Among all sequences $[T_n]\in
\overline{CV}(F_N)$ satisfying $\underset{n \to \infty}{\lim}[T_n]=[T_+]$ and
$\underset{n\ge 1}{\max}\,\, d_\mathcal I([T_n],[T])<\infty$, choose a sequence
$[T_n]$ such that $\underset{n\ge 1}{\max}\,\, d_\mathcal I([T_n],[T])$ is the
smallest possible.

Let $D=\underset{n\ge 1}{\max}\,\, d_\mathcal I([T_n],[T])$.  Suppose that $D>0$.
Then, after passing to a further subsequence, we may assume that
$[T_n]\ne [T]$ for every $n\ge 1$.  Note that by definition of the
graph $\mathcal I$, the numbers $D$ and $d_\mathcal I([T_n],[T])$ are
positive even integers.  By definition of $\mathcal I$ it follows that
there exist $[T_n']\in \overline{CV}(F_N)$ such that $d_\mathcal
I([T_n],[T_n'])=2$ and $d_\mathcal I([T_n'],[T])=d_\mathcal
I([T_n],[T])-2$.  Hence, again by definition of $\mathcal I$, there
exists a sequence $[\mu_n]\in \mathbb PCurr(F_N)$ such that $\langle
[T_n],[\mu_n]\rangle=0=\langle [T_n'],[\mu_n]\rangle$.  Since
$\underset{n \to \infty}{\lim}[T_n]=[T_+]$
and $\langle [T_n],[\mu_n]\rangle=0$,
Proposition~\ref{prop:key} implies that
$\underset{n \to \infty}{\lim}[\mu_n]=[\mu_+]$.  Since $\langle
[T_n'],[\mu_n]\rangle=0$, Proposition~\ref{prop:key} then implies that
$\underset{n \to \infty}{\lim}[T_n']=[T_+]$.  Thus $\underset{n \to \infty}{\lim}[T_n']=[T_+]$
and $\underset{n\ge 1}{\max}\,\, d_\mathcal I([T_n'],[T])=D-2<D=\underset{n\ge 1}{\max}\,\,
d_\mathcal I([T_n'],[T])$.  This contradicts the minimality in the
choice of $[T_n]$.  Therefore we conclude $D=0$.
Thus $0=\underset{n\ge 1}{\max}\,\, d_\mathcal I([T_n],[T])$ and hence $[T_n]=[T]$ for every $n\ge 1$.
This contradicts the assumptions that $[T]\ne [T_+]$ and that
$\underset{n \to \infty}{\lim}[T_n]=[T_+]$.
\end{proof}

Proposition~\ref{prop:inf} and Proposition~\ref{prop:ns} immediately
imply:

\begin{cor}\label{cor:iwip}
Let
$\varphi, [T_{+}]$ and $[T_{-}]$ be as in
Convention \ref{conv:iwip},
and let
$[T]\in \overline {CV}(F_N)$ be such that $[T]\ne [T_+], [T_-]$.
Then in the intersection graph $\mathcal I =\mathcal I(F_{N})$ we have:
$$\underset{n \to \infty}{\lim} d_{\mathcal
I}(\varphi^n[T],[T])=\infty$$
\end{cor}


\subsection{Other curve complex analogues}\ 

\medskip

One can define several other natural free group analogues of the curve
complex. Each of them will admit an $Out(F_N)$-equivariant Lipschitz
map into the intersection graph

\begin{defn}
\label{othercomplexes}
\rm
Let $N\ge 3$.
\begin{enumerate}
\item The \emph{free splitting graph} $\mathcal F=\mathcal{F}(F_N)$ is
  a simple graph whose vertices are non-trivial splitting of $F_N$ as
  the fundamental group of a graph-of-groups with a single non-loop
  edge and trivial edge group (so algebraically they correspond to
  nontrivial free product decompositions of $F_N$).  Two such
  splittings are considered equal if their Bass-Serre trees are
  $F_N$-equivariantly isometric, that is, if they equal as points of
  $cv(F_N)$. Two distinct splittings $T_1,T_2\in V\mathcal F(F_N)$ are
  adjacent in $\mathcal F$ if there exists a splitting of $F_N$ as the
  fundamental group of a graph-of-groups with two (non-loop) edges and
  trivial edge groups, such that $T_1$ is obtained by collapsing one
  edge of this graph-of-groups, and $T_2$ is obtained by collapsing
  the other edge.
\item The \emph{cut graph} $\mathcal S=\mathcal{S}(F_N)$ is
  a simple graph whose vertices are non-trivial splitting of $F_N$ as
  the fundamental group of a graph-of-groups with a single
  edge (possible a loop edge) and trivial edge group. Again, two such
  splittings are considered equal if their Bass-Serre trees are
  $F_N$-equivariantly isometric, that is, if they equal as points of
  $cv(F_N)$. Two distinct splittings $T_1,T_2\in V\mathcal F(F_N)$ are
  adjacent in $\mathcal F$ if there exists a splitting of $F_N$ as the
  fundamental group of a graph-of-groups with two (non-loop) edges and
  trivial edge groups, such that $T_1$ is obtained by collapsing one
  edge of this graph-of-groups, and $T_2$ is obtained by collapsing
  the other edge.
\item The \emph{dual free splitting graph} $\mathcal
  F^\ast=\mathcal{F}^\ast(F_N)$ has the same vertex set as
  $\mathcal{F}(F_N)$. Two distinct vertices of $\mathcal{F}^\ast(F_N)$
  corresponding to splittings $T_1$ and $T_2$ are adjacent whenever
  there exists a nontrivial element $g\in F_N$ such that
  $||g||_{T_1}=||g||_{T_2}=0$. One can show that this adjacency
  condition is equivalent to saying that there exists a nontrivial
  \emph{primitive} (i.e. a member of a free basis of $F_N$) element
  $a$ of $F_N$ such that $||a||_{T_1}=||a||_{T_2}=0$.
\item The \emph{dual cut graph}  $\mathcal
  S^\ast=\mathcal{S}^\ast(F_N)$ has the same vertex set as $\mathcal
  S(F_N)$. Two distince vertices are adjacent in
  $\mathcal{S}^\ast(F_N)$ whenever the corresponding splittings of
  $F_N$ have some common nontrivial elliptic element.
  
\item The \emph{ellipticity graph} $\mathcal Z=\mathcal Z (F_N)$
  is a bipartite graph. Its vertex set is the disjoint union of the
  vertex set of $\mathcal{S}(F_N)$ and the set of all $F_N$-conjugacy
  classes $[a]$ of nontrivial elements $a\in F_N$. A vertex $[a]$ is
  adjacent to a vertex $T$ whenever $a$ is an elliptic element with
  respect to $T$, that is $||a||_T=0$ (algebraically, this means that
  $a$ is conjugate to a vertex group element for the free product
  splitting $T$). 
\item The \emph{free factor graph} $\mathcal J=\mathcal J(F_N)$ is a
  simple graph defined as follows: The vertex set of of $\mathcal J$
  is the set of conjugacy classes in $F_N$ of all free factors $A$ of
  $F_N$ such that $A\ne 1, A\ne F_N$. Two distinct vertices $x,y\in
  V\mathcal J$ are adjacent in $\mathcal J$ if and only if for some
  $A,B$ with $[A]=x$ and $[y]=B$ there exists $C\le F_N$ such that
  $F_N=A\ast B\ast C$.  Note that we allow the case where $C=1$.

\item The \emph{dominance graph}  $\mathcal D=\mathcal
D(F_N)$ is defined as follows.  Put $V\mathcal D=V\mathcal J$.  For
distinct $x,y\in V\mathcal D$ we say that $x,y$ are adjacent in
$\mathcal D$ if and only if there exist $A,B\le F_N$ such that
$x=[A]$, $y=[B]$ and such that either $A\le B$ or $B\le A$. The
dominance graph is precisely the one-skeleton of the ``complex of free
factors'' $CF_N$ whose homotopy properties have been studied by
Hatcher and Vogtmann~\cite{HV}.

\item The \emph{primitivity graph} $\mathcal P(F_N)$ whose vertices
  are conjugacy classes of primitive elements of $F_N$ and where two
  such conjugacy classes are adjacent whenever there
  exist a free basis $X$ of $F_N$ and  some representatives $a_1$,
  $a_2$ of these conjugacy classes such that $a_1,a_2\in X$. 
\end{enumerate}
\end{defn}

It is not hard to see that for $N\ge 3$ all of these graphs are
connected and they come equipped with natural $Out(F_N)$-actions.
Moreover, with a bit of work, one can show that for a given $N\ge 3$
there are at most two substantially distinct objects among the above
graphs. 

More specifically,
the graphs $\mathcal F(F_N)$ and $\mathcal S(F_N)$ are
quasi-isometric (this is almost immediate from the definitions).
Moreover, for $N\ge 3$ the graphs $\mathcal F^\ast(F_N)$, $\mathcal
S^\ast(F_N)$, $\mathcal Z(F_N)$, $\mathcal J(F_N)$, $\mathcal D(F_N)$
and $\mathcal P(F_N)$ are all quasi-isometric.  Also, the full
subgraph of $\mathcal Z(F_N)$ induced by all the vertices coming from
$\mathcal{S}(F_N)$ together with conjugacy classes of primitive
elements of $F_N$ can be shown to be a co-bounded subset of $\mathcal
Z(F_N)$ and is thus quasi-isometric to $\mathcal Z(F_N)$.

Note that there are natural
$Out(F_N)$-equivariant Lipshitz maps $j:\mathcal F(F_N)\to \mathcal
F^*(F_N)$ and $q: \mathcal F^*(F_N)\to \mathcal I_0(F_N)$. The map $j$
is the identity map on the vertex set of $\mathcal F(F_N)$ (recall
that by definition the vertex sets of $\mathcal F(F_N)$ and $\mathcal
F^\ast(F_N)$ are equal).  The map $q$ sends a vertex $T$ of $\mathcal
F^\ast(F_N)$ to a vertex $[T]$ of $\mathcal I_0(F_N)$.

Note that if $T_1$ and $T_2$
are adjacent in $\mathcal F(F_N)$ and the tree $T$ corresponds to
their common refinement, then for any nontrivial elliptic element $a$
for $T$ we have $||a||_T=||a||_{T_1}=||a||_{T_2}=0$. Hence $T_1$ and
$T_2$ are adjacent in $\mathcal F^\ast(F_N)$ and therefore the map $j$
is $1$-Lipshitz.

Similarly, suppose that two vertices $T_1$ and $T_2$ of $\mathcal
F^\ast(F_N)$ are adjacent in $\mathcal
F^\ast(F_N)$. Hence there exists a nontrivial element $a\in F_N$ such
that $||a||_{T_1}=||a||_{T_2}=0$. Hence by the properties of the
intersection form $\langle T_i,\eta_a\rangle=||a||_{T_i}=0$ for
$i=1,2$. Hence both $[T_1]$ and $[T_2]$ are adjacent to $[\eta_a]$ in
$\mathcal I(F_N)$ which implies that the map $q$ is
$2$-Lipshitz.

It appears (although we do not know how to prove this) that the map $j$, although Lipshitz, is not a
quasi-isometry, and that the fibers of $j$ have infinite diameter as
subsets of  $\mathcal F^*(F_N)$.

Since the maps $q$, $j$ and $q\circ j$ are Lipshitz,
Corollary~\ref{cor:iwip} immediately implies analogous statements for
the above graphs:

\begin{cor}\label{cor:diam}
Let $N\ge 3$ and let $\varphi\in Out(F_N)$ be an atoriodal iwip.
Let $\mathcal Y$ be one of the graphs $\mathcal F(F_N)$, $\mathcal
F^*(F_N)$.  Then for any vertex
$x$ of $\mathcal Y$ we have:
\[\underset{n \to \infty}{\lim} d_{\mathcal
Y}(\varphi^nx,x)=\infty.\]
In particular, $diam(\mathcal Y)=\infty$.
\end{cor}

The above corollary, together with Corollary~\ref{cor:iwip}, implies
Theorem~\ref{thm:B} from the introduction.


\end{document}